\documentclass[10pt]{amsart}
\usepackage[utf8]{inputenc}
\usepackage[english]{babel}
\usepackage{amsmath,amsfonts,hyperref,amsthm,tikz-cd,hyperref,microtype,csquotes,amssymb,enumitem}

\newcommand{\Ker}{\mathrm{Ker}}

\newcommand{\St}{\mathrm{St}\,}

\newcommand{\Spec}{\mathrm{Spec}}

\newcommand{\sgn}{\mathrm{sgn}}

\makeatletter
\def\th@plain{%
  \thm@notefont{}
  \normalfont 
}
\def\th@definition{%
  \thm@notefont{}
  \normalfont 
}
\makeatother

\theoremstyle{plain}

\newtheorem{theorem}{Theorem}[section]
\newtheorem{proposition}[theorem]{Proposition}
\newtheorem{lemma}[theorem]{Lemma}

\theoremstyle{definition}

\newtheorem{definition}[theorem]{Definition}
\newtheorem{remark}[theorem]{Remark}

\author{Igor Baskov}
\title{The de~Rham cohomology\\of the algebra of polynomial functions\\ on a simplicial complex}
\thanks{This research was supported by the Ministry of Science and Higher Education of the
Russian Federation, agreement 075-15-2019-1620 date 08/11/2019 and 075-15-2022-289 date 06/04/2022}
\usetikzlibrary{babel}
\newcommand{\Addresses}{{
  \bigskip
  \footnotesize
\noindent
  \textsc{St. Petersburg Department of Steklov Mathematical Institute\\of Russian Academy of Sciences}\par\noindent
  \textit{E-mail address}: \texttt{baskovigor@pdmi.ras.ru}
}}

\begin{document}
\maketitle
\begin{abstract}
We consider the algebra $A^0 (X)$ of polynomial functions on a simplicial complex $X$.
The algebra $A^0 (X)$ is the $0$th component of Sullivan's dg-algebra $A^\bullet (X)$ of polynomial forms on $X$.
All algebras are over an arbitrary field $k$ of characteristic $0$.

Our main interest lies in computing the de Rham cohomology of the algebra $A^0(X)$, that is, the cohomology of the universal dg-algebra $\Omega ^\bullet _{A^0(X)}$. 
There is a canonical morphism of dg-algebras $P:\Omega ^\bullet _{A^0(X)} \to A^\bullet (X)$.
We prove that $P$ is a quasi-isomorphism.
Therefore, the de Rham cohomology of the algebra $A^0 (X)$ is canonically isomorphic to the cohomology of the simplicial complex $X$ with coefficients in $k$.
Moreover, for $k=\mathbb{Q}$ the dg-algebra $\Omega ^\bullet _{A^0 (X)}$ is a model of the simplicial complex $X$ in the sense of rational homotopy theory.
Our result shows that for the algebra $A^0 (X)$ the statement of Grothendieck's comparison theorem holds.
\end{abstract}
\section{Introduction}
All algebras and dg-algebras are commutative over a field $k$ of characteristic $0$.
In \cite[Section~$7$]{sullivan1977infinitesimal} Sullivan introduces the dg-algebra $A^\bullet (X)$ of \textit{polynomial forms} on a simplicial complex $X$.
The algebra $A^0(X)$ of the degree zero elements of $A^\bullet (X)$ is the algebra of \textit{polynomial functions} on $X$. 
The cohomology of the dg-algebra $A^\bullet (X)$ is isomorphic to $H^\bullet(X,k)$.
One can ask what natural dg-algebras are weakly equivalent to $A^\bullet (X)$.
One such candidate is the universal dg-algebra $\Omega ^\bullet _{A^0 (X)}$ on the algebra $A^0 (X)$ of polynomial functions on $X$.
There is a canonical morphism of dg-algebras $P:\Omega ^\bullet _{A^0 (X)}\to A^\bullet (X)$.

The main result is Theorem~\ref{cor:mainersult}, where we prove that $P$ is a quasi-isomorphism.

In \cite{kan1976sullivan} the authors prove that the morphism $P$ is surjective and give a description of its kernel.
In \cite{felix2009combinatorial} and \cite[Appendix G(i)]{sullivan1973differential} another description of the kernel is given.
In \cite[Example~$3.8$]{gomez2002simplicial}, G{\'o}mez establishes that the morphism $P$ is \textit{not} a quasi-isomorphism, which contradicts our main result, Theorem~\ref{cor:mainersult}.
We were able to correct the erroneous computation of G{\'o}mez in Remark~\ref{rem:gomezexample}.

Grothendieck proved that for a smooth $\mathbb{C}$-algebra $A$ the cohomology groups of the algebraic de Rham complex $\Omega ^\bullet _A$ are isomorphic to the cohomology groups of the space $\mathop{\Spec} A$ with complex analytic topology, see \cite[Theorem~$1$]{grothendieck1966rham}.
The algebra $A^0 (X)$ is not smooth in general and the result of Grothendieck does not hold for general algebras, see \cite[Example~$4.4$]{arapura2011kahler}.

The result of this paper can be used in order to give another proof of the similar result for the algebra of piecewise polynomial functions on a polyhedron, which is known due to \cite[Theorem~$51$]{ib}.
\subsection*{Acknowledgement}
I would like to thank Dr. Sem\"en Podkorytov for his immense help, fruitful discussions and the idea to consider \v Cech resolution in order to prove the result.
I am grateful to the St. Petersburg Department of Steklov Mathematical Institute of Russian Academy of Sciences for their financial assistance.

\section{Simplicial complexes}
\begin{definition}
We call a set $X$ of finite non-empty subsets of a finite set $E$ a \textit{simplicial complex} if for every $v\in E$ we have $\lbrace v\rbrace \in X$ and for every $s\in X$ and every non-empty subset $s^\prime\subset s$ we have $s^\prime\in X$.
We denote by $V(X)$ the set $E$ and call its elements \textit{vertices} of $X$.
The sets $s\in X$ of cardinality $m+1$ are called $m$-simplices.
A simplicial complex $Y$ is a \textit{subcomplex} of $X$ if for every $s\in Y$ we have $s\in X$.
\end{definition}

We denote by $T_p (X)$ the set of all sequences $u=(u_0,\dots,u_p)$ of vertices of $X$ for $p\geq -1$.
We denote by $\partial _i u$ the sequence $(u_0,\dots,\widehat{u_i},\dots,u_p)$.
For a vertex $v$ we denote by $v*u$ the sequence $(v,u_0,\dots,u_p)$.
The symmetric group $\Sigma _{p+1}$ acts on $T_p(X)$.

Consider a sequence of vertices $u\in T_p(X)$.
We denote by $\St u$ the \textit{star} of $u$, that is the smallest subcomplex of $X$ containing all the simplices containing the vertices $u_i$.
If $p=-1$, we have $\St u = X$.
If the sequence $u$ spans a simplex in $X$, then $\St u$ is the star of this simplex.
If $p\geq 0$ and the sequence $u$ does not span a simplex, we have $\St u =\varnothing$.
For a subcomplex $Y$ of $X$ we denote by $\St _Y u$ the smallest subcomplex of $Y$ containing all the simplices containing the vertices $u_i$.
If $u\notin T_p(Y)$ then $\St _Y u=\varnothing$.

\section{Sullivan's dg-algebra of polynomial forms}

For a simplicial complex $X$ we define the dg-algebra $A^\bullet (X)$ following Sullivan, see \cite[Section~$7$]{sullivan1977infinitesimal},  \cite{morgan}.
For an $m$-simplex $a$ consider the dg-algebra
$$A^\bullet (a):=\dfrac{\Lambda \left( t_v, dt_v \mid \mathrm{deg} (t_v) = 0, \, \mathrm{deg} (dt_v) = 1,\, v\in a\right)}{\left( \sum _{v\in a} t_v - 1, \sum _{v\in a} dt_v \right)}$$
with the differential $t_v \mapsto dt_v$ for $v\in a$.

For a simplex $b$ such that $b\subset a$ one has a natural morphism of dg-algebras
$$|_b :A^\bullet (a)\to A^\bullet (b), \qquad
t_v\mapsto
\begin{cases}
0,\, v\notin b,\\
t_v,\, v\in b.
\end{cases}$$
Then an element $\omega = (\omega _a)_{a\in X}$ of $A^\bullet (X)$ is a collection of elements $\omega _a \in A^\bullet (a)$ such that for two simplices $b\subset a$ one has $\omega _a|_b=\omega _b$.

We call the algebra $A^0 (X)$ the \textit{algebra of polynomial functions} on $X$.
This algebra has another description as a quotient of a Stanley-Reisner algebra, see \cite{billera1989algebra}.

An inclusion of simplicial complexes $Y\subset X$ gives rise to the restriction morphism of dg-algebras
$${}|_{Y}:A^\bullet (X)\to A^\bullet (Y).$$
\begin{lemma}\label{lem:therestrictionissurjective}
The above restriction morphism is surjective.
\end{lemma}
\begin{proof}
This fact is quite nontrivial, see \cite[Section~$7$]{sullivan1977infinitesimal}.

\end{proof}

We introduce the double graded vector space $\mathcal{D} ^{p,q}$, $p,q\in \mathbb{Z}$ as follows.
For $p\leq -2$ set $\mathcal{D} ^{p,q} = 0$ and for $p\geq -1$, we define $\mathcal{D} ^{p,q}$ as the subspace of
$$\prod _{u\in T_p(X)} A ^q (\St u)$$
consisting of families of forms $\omega _u\in A ^q (\St u)$, such that for any $\sigma\in\Sigma_{p+1}$ we have
$$\omega _{\sigma u} = (\sgn \mathop \sigma) \omega _u.$$

We define the linear map
$$\delta :\mathcal{D} ^{p,q}\to \mathcal{D} ^{p+1,q}.$$
For $p\geq -1$ and for $$\omega=(\omega _u)_{u\in T_p(X)}\in\mathcal{D} ^{p,q}$$
we set the value of $\delta\omega$ on $s \in T_{p+1}(X)$ as
$$(\delta \omega )_{s} = \sum _{i=0} ^{p+1} (-1)^i \omega _{\partial _i s}|_{\St s}.$$
The differential $d$ on $A ^\bullet (\St u)$ gives rise to a differential $d$ on $\mathcal{D} ^{p,\bullet}$ for each $p$.

\begin{proposition}
The map $\delta$ is a differential on $\mathcal{D} ^{\bullet ,q}$ for each $q\in\mathbb{Z}$.
Moreover, the double graded vector space $\mathcal{D} ^{\bullet, \bullet}$ together with $\delta$ and $d$ forms a double complex in the sense that $d\,\delta = \delta\, d$.
\end{proposition}

\begin{proposition}\label{thm:Acomplexisexact}
The complex
\begin{equation*}
\begin{tikzcd}
0\arrow[r]&\mathcal{D} ^{-1,q}\arrow[r,"\delta"]\arrow[equals]{d}&\mathcal{D} ^{0,q}\arrow[r,"\delta"]&\mathcal{D} ^{1,q}\arrow[r,"\delta"]&\dots\\
&A^q (X)&&&
\end{tikzcd}
\end{equation*}
is exact.
\end{proposition}

\begin{proof}
The proof is similar to that of Proposition~\ref{thm:maincomplexisexact} below and relies on Lemma~\ref{lem:therestrictionissurjective}.
In this case one can use the partition of unity $t_v,\, v\in V(X)$, instead of $\rho _v, v\in V(X)$.
 \end{proof}

\section{The dg-algebra of de~Rham forms}

To a $k$-algebra $A$ one associates the commutative dg-algebra $\Omega ^\bullet _{A}$ (\cite[Theorem~$3.2$]{KunzKahlerDifferentials}) with $\Omega ^0 _{A} = A$.
It has the following universal property: for any dg-algebra $E^\bullet$ and any algebra homomorphism $f:A\to E^0$ there exists a unique morphism of dg-algebras $F:\Omega ^\bullet _{A}\to E^\bullet$ such that $F|_A = f$:
\begin{equation*}
\begin{tikzcd}
A\arrow[r]\arrow[rd,"f"]&\Omega ^\bullet _{A}\arrow[d,dashed,"F"]\\
&E^\bullet .
\end{tikzcd}
\end{equation*}
The elements of $\Omega ^q _{A}$ are called algebraic $q$-forms.
The dg-algebra $\Omega ^\bullet _{A}$ is covariant in the algebra $A$.
We will simply write $\Omega ^\bullet (X)$ for the dg-algebra $\Omega ^\bullet _{A^0 (X)}$.

Inclusion of simplicial complexes $Y\subset X$ gives rise to the restriction morphism of dg-algebras
$${}|_{Y}:\Omega ^\bullet (X)\to \Omega ^\bullet (Y).$$

\begin{lemma}\label{lem:kernelofinducedmapofomegas}
Suppose $A$ and $B$ are $k$-algebras and $\varphi:A\to B$ is a surjective homomorphism of algebras.
Then the induced morphism $\Omega _\varphi:\Omega ^\bullet _{A}\to\Omega ^\bullet _{B}$ is surjective and its kernel is the ideal of $\Omega ^\bullet _{A}$ generated by $\Ker \varphi$ and $d(\Ker \varphi)$.
\end{lemma}
From this and Lemma~\ref{lem:therestrictionissurjective} it follows that for an inclusion of simplicial complexes $Y\subset X$ the restriction morphism ${}|_{Y}:\Omega ^\bullet (X)\to \Omega ^\bullet (Y)$ is surjective.

\begin{proof}
See \cite[Lemma~$6$]{ib}.
 \end{proof}

Let us introduce the following elements $t_v$ of $A^0 (X)$.
For a vertex $v\in V(X)$ and an $m$-simplex $a\in X$ set $(t_v)_a = 0$ if $v\notin a$ and $(t_v)_a = t_v\in A^0 (a)$ if $v\in a$. 

\begin{lemma}\label{lem:nulityofomega}
Take a simplicial complex $X$ and a subcomplex $Y\subset X$.
Suppose $\omega\in \Omega ^q (Y)$ is such that $\omega |_{\St _Y (v)} = 0$ for a vertex $v\in V(X)$.
Then $t_v ^2 |_Y \, \omega = 0$.
\end{lemma}

\begin{proof}
First, if $v\notin V(Y)$, then $t_v|_Y = 0$ and the claim follows.

Assume $v\in V(Y)$.
By Lemma~\ref{lem:therestrictionissurjective} and Lemma~\ref{lem:kernelofinducedmapofomegas}, the form $\omega$ lies in the dg-ideal $I$ of $\Omega ^\bullet (X)$ generated by the elements $m\in A^0(X)$ with the restriction to $\St _Y(v)$ being zero, therefore $t_v|_Y \, m=0$.
It is enough to consider the cases $\omega = m$ and $\omega = dm$.
We have $t_v ^2 |_Y \, m = 0$ and $t_v^2 |_Y \, d m = t_v |_Y \, d(t_v|_Y \, m)-t_v|_Y \, mdt_v|_Y = 0$.
 \end{proof}

We introduce the double graded vector space $\mathcal{C} ^{p,q}$, $p,q\in \mathbb{Z}$ as follows.
For $p\leq -2$ set $\mathcal{C} ^{p,q} = 0$ and for $p\geq -1$, we define $\mathcal{C} ^{p,q}$ as the subspace of
$$\prod _{u\in T_p(X)} \Omega ^q (\St u)$$
consisting of families of forms $\omega _u\in \Omega ^q (\St u)$, such that for any $\sigma\in\Sigma_{p+1}$ we have
$$\omega _{\sigma u} = (\sgn \mathop \sigma) \omega _u.$$

We define the linear map
$$\delta :\mathcal{C} ^{p,q}\to \mathcal{C} ^{p+1,q}.$$
For $p\geq -1$ and for $$\omega=(\omega _u)_{u\in T_p(X)}\in\mathcal{C} ^{p,q}$$
we set the value of $\delta\omega$ on $s \in T_{p+1}(X)$ as
$$(\delta \omega )_{s} = \sum _{i=0} ^{p+1} (-1)^i \omega _{\partial _i s}|_{\St s}.$$
The differential $d$ on $\Omega ^\bullet (\St u)$ gives rise to a differential $d$ on $\mathcal{C} ^{p,\bullet}$ for each $p$.

\begin{proposition}
The map $\delta$ is a differential on $\mathcal{C} ^{\bullet ,q}$ for each $q\in\mathbb{Z}$.
Moreover, the double graded vector space $\mathcal{C} ^{\bullet, \bullet}$ together with $\delta$ and $d$ forms a double complex in the sense that $d\,\delta = \delta\, d$.
\end{proposition}

\begin{lemma}
There exist elements $p_v\in A^0(X), \, v\in V(X)$, such that
$$\sum _{v\in V(X)} p_v t_{v} ^2 =1.$$
\end{lemma}
\begin{proof}
In $A^0(X)$ we have the equality
$$\sum _{v\in V(X)} t_{v} =1.$$
Raise both the sides to a big enough power and obtain the needed equality.
 \end{proof}

We put $\rho _v = p_v t_v ^2$.

For an inclusion of simplicial complexes $Y \subset X$ we choose a linear map, the distinguished ``extension'',
$$[- ]:\Omega ^q (Y) \to \Omega ^q (X),$$
such that $[\omega ]|_{Y}=\omega$.
Such an extension exists by Lemma~\ref{lem:therestrictionissurjective} applied to $A^0 (X)$ and Lemma~\ref{lem:kernelofinducedmapofomegas}.

\begin{lemma}\label{lem:extres}
For an inclusion of simplicial complexes $Y\subset X$ and a form $\omega\in \Omega ^q (Y)$ we have
$$\sum _{v\in V(X)} \rho _v [ \omega |_{\St _Y (v)} ]|_Y  = \omega .$$
\end{lemma}
\begin{proof}
As $\sum _v \rho _v =1$ in $A^0(X)$ we have
$$\sum _{v\in V(X)} \rho _v [ \omega |_{\St _Y (v)} ]|_Y -\omega = \sum _{v\in V(X)} \rho _v |_Y\left([ \omega |_{\St _Y (v)} ]|_Y -\omega\right).$$

We have 
$$\left([ \omega |_{\St _Y (v)} ]|_Y -\omega\right)|_{\St _Y (v)}=\omega |_{\St _Y (v)}-\omega |_{\St _Y (v)}=0.$$
Hence, by Lemma~\ref{lem:nulityofomega} we have $\rho _v |_Y\left([ \omega |_{\St _Y (v)} ]|_Y -\omega\right)=0$.
 \end{proof}

\begin{proposition}\label{thm:maincomplexisexact}
The complex
\begin{equation*}
\begin{tikzcd}
0\arrow[r]&\mathcal{C} ^{-1,q}\arrow[r,"\delta"]\arrow[equals]{d}&\mathcal{C} ^{0,q}\arrow[r,"\delta"]&\mathcal{C} ^{1,q}\arrow[r,"\delta"]&\dots\\
&\Omega ^q (X)&&&.
\end{tikzcd}
\end{equation*}
is exact.
\end{proposition}

The proof follows the proof of \cite[Proposition~$8.5$]{bott1982differential}.

\begin{proof}
First, notice that
$$\St v*u = \St _{\St u} (v).$$

For $u\in T_p(X)$ and $\omega \in \Omega ^q (\St u)$ by Lemma~\ref{lem:extres} we have
\begin{equation}
\sum _{v\in V(X)} \rho _v [ \omega |_{\St v*u} ]|_{\St u} = \omega .
\end{equation}

For $\omega\in \Omega ^q (Z)$, where $Z$ is a subcomplex of $X$ such that $\St v*u\subset Z$, by Lemma~\ref{lem:nulityofomega}, we have 
\begin{equation}
\rho _v \left([\omega ] - [\omega |_{\St v*u} ]\right) |_{\St u} = 0.
\end{equation}

We construct a cochain homotopy
$$K:\mathcal{C} ^{p,q} \to \mathcal{C} ^{p-1,q}.$$

For $p\geq 0$ and $\omega\in\mathcal{C} ^{p,q}$ and $w\in T_{p-1} (X)$ put
$$(K\omega)_{w}:=\sum _{v\in V(X)} \rho _v [\omega _{v*w}]|_{\St w}.$$
By Lemma~\ref{lem:nulityofomega} this map does not depend on the choice of the distinguished extension.

Let us check that $\delta K + K\delta =1$.
For $p\geq -1$ and
$$\omega = (\omega _u )_{u\in T_p(X)}\in\mathcal{C} ^{p,q}\subset\prod _{u\in T_p(X)} \Omega ^q (\St u),$$
where $\omega _u\in \Omega ^q (\St u)$, we have
$$(\delta K \omega)_{u}=\sum _{i=0} ^p (-1)^i (K\omega)_{\partial _i u}|_{\St u} = \sum _{i=0} ^p (-1)^i \sum _{v\in V(X)} \rho _v [\omega _{v*\partial _i u}] |_{\St u},$$
and
$$(K\delta \omega)_{u} = \sum _{v\in V(X)} \rho _v [ (\delta \omega)_{v*u} ]|_{\St u} =$$
$$= \sum _{v\in V(X)} \rho _v [\, \sum _{i=0} ^{p+1} (-1)^i\omega_{\partial _i (v*u)}|_{\St v*u} ] |_{\St u} = $$
$$ = \sum _{v\in V(X)} \rho _v [\omega_{u}|_{\St v*u} ] |_{\St u}+\sum _{v\in V(X)} \rho _v [\, \sum _{i=1} ^{p+1} (-1)^i\omega_{\partial _i (v*u)}|_{\St v*u} ] |_{\St u} \overset{\text{by (1)}}{=} $$
$$ = \omega_u- \sum _{v\in V(X)} \rho _v [\, \sum _{i=0} ^p (-1)^i\omega_{v*\partial _i u}|_{\St v*u} ] |_{\St u} $$
$$ = \omega_u - \sum _{i=0} ^p (-1)^i\sum _{v\in V(X)}\rho _v [\omega_{v*\partial _i u}|_{\St v*u} ] |_{\St u}.$$

Hence,
$$((\delta K + K \delta ) \omega)_{u} = \omega_{u} + \sum _{i=0} ^p (-1)^i \sum _{v\in V(X)} \rho _v \left( [\omega _{v*\partial _i u }] - [\omega_{v*\partial _i u}|_{\St v*u} ]\right)|_{\St u}\overset{\text{by (2)}}{=}\omega_{u}.$$
 \end{proof}

\section{The morphism $P:\Omega ^\bullet (X)\to A ^\bullet (X)$}

For a simplicial complex $X$, by the universal property of $\Omega ^\bullet (X)$, there is a canonical morphism dg-algebras
$$P:\Omega ^\bullet (X)\to A ^\bullet (X),$$
which is the identity in degree $0$.

We denote by $k[0]$ the complex with the $0$th term $k$ and the others zero.
An element of $k$ gives rise to a constant function in $A^0 (X)$, hence, there are morphisms of complexes $\epsilon :k[0]\to \Omega ^\bullet (X)$ and $\bar{\epsilon} :k[0]\to A ^\bullet (X)$ such that $\bar{\epsilon} = P\circ\epsilon$.
\begin{proposition}\label{lem:Poincarelemma}
For a sequence $u\in T_p(X)$, $p\geq 0$, the commutative diagram
\begin{equation*}
\begin{tikzcd}
k[0]\arrow[r,"\epsilon"]\arrow[rd,"\bar{\epsilon}"]&\Omega ^\bullet (\St u)\arrow[d,"P"]\\
&A ^\bullet (\St u)
\end{tikzcd}
\end{equation*}
consists of quasi-isomorphisms.
\end{proposition}
\begin{proof}
The map $\epsilon$ is a quasi-isomorphism by \cite[Corollary~$47$]{ib}.
The map $\bar{\epsilon}$ is a quasi-isomorphism by \cite[Theorem~$10.9$]{felix2012rational}.
Hence, the morphism $P$ is a quasi-isomorphism.
 \end{proof}

\begin{theorem}\label{cor:mainersult}
The natural map
$$P:\Omega ^\bullet (X)\to A ^\bullet(X)$$
is a quasi-isomorphism.
\end{theorem}
\begin{proof}
The morphism $P$ on stars gives rise to the maps $\pi _p:\mathcal{C} ^{p,\bullet}\to \mathcal{D} ^{p,\bullet}$ for each $p\geq 0$.
We have the following commutative diagram of non-negative complexes
\begin{equation*}
\begin{tikzcd}
0\arrow[r]&\Omega ^\bullet (X)\arrow[r,"\delta"]\arrow[d,"P"]&\mathcal{C} ^{0,\bullet}\arrow[r,"\delta"]\arrow[d,"\pi _0"]&\mathcal{C} ^{1,\bullet}\arrow[r,"\delta"]\arrow[d,"\pi _1"]&\dots\\
0\arrow[r]&A ^\bullet(X)\arrow[r,"\delta"]&\mathcal{D} ^{0,\bullet}\arrow[r,"\delta"]&\mathcal{D} ^{1,\bullet}\arrow[r,"\delta"]&\dots .
\end{tikzcd}
\end{equation*}
The vertical arrows $\pi _p$, $p\geq 0$ are quasi-isomorphisms by Proposition~\ref{lem:Poincarelemma}.
The first row is exact by Proposition~\ref{thm:maincomplexisexact}.
The second row is exact by Proposition~\ref{thm:Acomplexisexact}.

Hence, the map $P$ is also a quasi-isomorphism.
\end{proof}

\begin{remark}\label{rem:gomezexample}
As was said in the introduction, the paper \cite{gomez2002simplicial} suggests that Theorem~\ref{cor:mainersult} is false.
Namely, in \cite[Example~$3.8$]{gomez2002simplicial}, one considers the simplicial complex $X$ corresponding to the boundary of a triangle on the vertices $1$, $2$, $3$.
The dg-algebra $\Omega ^\bullet (X)$ is generated by the elements $t_1$, $t_2$, $t_3$ modulo the dg-ideal generated by $t_1+t_2+t_3-1$ and $t_1t_2t_3$.
Next, the author considers the form $t_1 ^2 t_2 ^2 dt_3$ and claims that this form is not zero.
However, this form is zero, which can be seen as follows:
applying the differential $d$ to the equality $t_1t_2t_3=0$ we get
$$t_2t_3dt_1+t_1t_3dt_2+t_1t_2dt_3=0.$$
Next, we multiply this equality by $t_1t_2$ and get
$$0=t_1t_2 ^2t_3dt_1+t_1 ^2 t_2t_3dt_2+t_1 ^2t_2 ^2dt_3 = t_1 ^2t_2 ^2dt_3.$$
\end{remark}
\bibliographystyle{siam}
\bibliography{sample}

\begin{thebibliography}{10}

\bibitem{arapura2011kahler}
{\sc D.~Arapura and S.-J. Kang}, {\em K{\"a}hler--de {Rham} cohomology and
  {Chern} classes}, Communications in Algebra, 39 (2011), pp.~1153--1167.

\bibitem{ib}
{\sc I.~Baskov}, {\em The de {Rham} cohomology of soft function algebras}.
\newblock \url{https://arxiv.org/abs/2208.11431}, 2022.

\bibitem{billera1989algebra}
{\sc L.~J. Billera}, {\em The algebra of continuous piecewise polynomials},
  Advances in Mathematics, 76 (1989), pp.~170--183.

\bibitem{bott1982differential}
{\sc R.~Bott and L.~W. Tu}, {\em Differential forms in algebraic topology},
  Springer, 1982.

\bibitem{felix2012rational}
{\sc Y.~F{\'e}lix, S.~Halperin, and J.-C. Thomas}, {\em Rational homotopy
  theory}, Springer, 2012.

\bibitem{felix2009combinatorial}
{\sc Y.~F{\'e}lix, B.~Jessup, and P.-E. Parent}, {\em The combinatorial model
  for the {Sullivan} functor on simplicial sets}, Journal of Pure and Applied
  Algebra, 213 (2009), pp.~231--240.

\bibitem{gomez2002simplicial}
{\sc F.~G{\'o}mez}, {\em Simplicial types and polynomial algebras}, Archivum
  Mathematicum, 38 (2002), pp.~27--36.

\bibitem{morgan}
{\sc P.~Griffiths and J.~Morgan}, {\em Rational homotopy theory and
  differential forms}, Birkhäuser, 1981.

\bibitem{grothendieck1966rham}
{\sc A.~Grothendieck}, {\em On the de {Rham} cohomology of algebraic
  varieties}, Publications Math{\'e}matiques de l'IH{\'E}S, 29 (1966),
  pp.~95--103.

\bibitem{kan1976sullivan}
{\sc D.~M. Kan and E.~Y. Miller}, {\em Sullivan’s de {Rham} complex is
  definable in terms of its 0-forms}, Proceedings of the American Mathematical
  Society, 57 (1976), pp.~337--339.

\bibitem{KunzKahlerDifferentials}
{\sc E.~Kunz}, {\em K{\"a}hler differentials}, Friedr. Vieweg \& Sohn, 1986.

\bibitem{sullivan1973differential}
{\sc D.~Sullivan}, {\em Differential forms and the topology of manifolds},
  Manifolds Tokyo,  (1973), pp.~37--49.

\bibitem{sullivan1977infinitesimal}
\leavevmode\vrule height 2pt depth -1.6pt width 23pt, {\em Infinitesimal
  computations in topology}, Publications Math{\'e}matiques de l'IH{\'E}S, 47
  (1977), pp.~269--331.

\end{thebibliography}
\Addresses
\label{end}
\end{document}